\documentclass{amsproc}

\usepackage{graphicx}

\usepackage{}

\newtheorem{theorem}{Theorem}[section]

\newtheorem{proposition}[theorem]{Proposition}
\newtheorem{cor}[theorem]{Corollary}

\theoremstyle{definition}
\newtheorem{definition}[theorem]{Definition}

\theoremstyle{remark}

\numberwithin{equation}{section}

\begin{document}

\title{A Paradoxical Decomposition of the Real Line}


\author{Shelley Kandola}
\address{Shelley Kandola, St. Lawrence University, 23 Romoda Drive, SMC \# 2218, Canton, NY 13617 USA}
\curraddr{Department of Mathematics, University of Minnesota, 206 Church Street, Minneapolis, MN 55455 USA}
\email{kando004@math.umn.edu}
\thanks{}

\author{Sam Vandervelde}
\address{Sam Vandervelde, Dept of Math, CS \& Stats, St. Lawrence University, 23 Romoda Drive, Canton NY 13617 USA}
\curraddr{Proof School, 555 Post Street, San Francisco, CA 94102 USA}
\email{svandervelde@proofschool.org}
\thanks{}


\date{}

\begin{abstract}
In this paper we demonstrate how to partition the real number line into four subsets which may be reassembled, via ``piecewise rigid functions'' that preserve Lebesgue measure, into two copies of the line. We then employ a similar process to split the line into $2k$ pieces that yield $k$ copies of the line, or even into countably many subsets to obtain countably many copies of $\mathbb{R}$.
\end{abstract}

\maketitle

\section{Motivation}
In its most familiar form, the Banach-Tarski paradox \cite{B24} asserts that it is possible
to decompose a sphere into finitely many disjoint subsets that may be reassembled
via translations and rotations into two spheres congruent to the original one, with
no overlap or missing points. It is also known (see \cite{W99}) that such a sleight of hand
is not possible for point sets within the real number line, at least not if we restrict
ourselves to finitely many subsets acted upon by isometries.

Wagon \cite{W99} points out that if one enlarges the group of available actions
to include all bijections preserving Lebesgue measure, then it does become possible
to obtain a paradoxical decomposition of an interval, for instance. However, the
means for exhibiting such a decomposition is not altogether satisfactory, in that it
relies upon a black box; namely, a measure-preserving bijection from $[0, 1)$ to $S^2$,
coupled with the Banach-Tarski result.

The purpose of this note is to present an explicit example of such a paradoxical
decomposition of the real number line, using a minimum of external machinery.
Since actions by isometries alone are insufficient to produce a paradox, we employ
functions that are ``piecewise rigid.''

\begin{definition}
A function $f:\mathbb{R} \to \mathbb{R}$ is {\emph{piecewise rigid}} if
\begin{enumerate}
\item $f$ is a bijection,
\item $f$ is differentiable with $|f'(x)|=1$, except at a countable number of jump discontinuities, and
\item the set of discontinuities of $f$ has no limit points in $\mathbb{R}$.
\end{enumerate}
\end{definition}

Clearly piecewise rigid functions preserve Lebesgue measure. We mention without proof that the set of all piecewise rigid function on $\mathbb{R}$ form a group with respect to composition. We can now describe our goal more precisely: we will partition the real number line into four subsets in such a way that it is possible to ``reassemble'' the subsets using piecewise rigid functions to produce two copies of the line.

\section{From Permutations to Free Groups}
A key ingredient in the construction of the paradoxical subsets consists of a pair of functions that generate a free group. Therefore our primary task is to produce such functions. We do so by observing that given any permutation of the integers, we may extend it to a piecewise rigid bijection of the real numbers.

\begin{definition}
Given a permutation $\sigma:\mathbb{Z} \to \mathbb{Z}$, we define $f_\sigma:\mathbb{R} \to \mathbb{R}$ by setting $f_\sigma(x) = \sigma(\lfloor x\rfloor)+\langle x \rangle$, where $\lfloor x \rfloor$ is the greatest integer function and $\langle x \rangle$ is the fractional part of $x$, so that $x=\lfloor x \rfloor + \langle x \rangle$.
\end{definition}

\begin{figure}
\includegraphics[scale=.5]{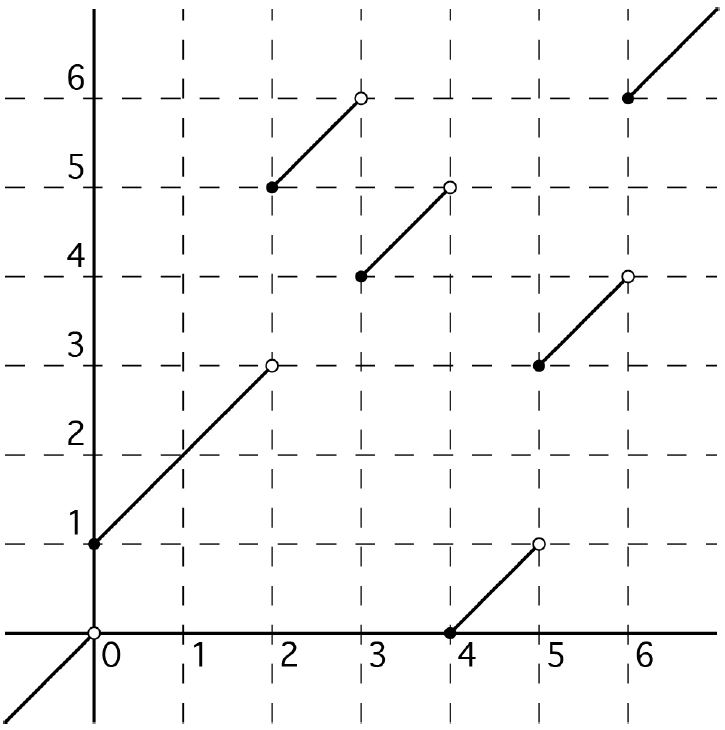}
\caption{The graph of the function $f_\sigma$ corresponding to $\sigma = (012534)$.}
\label{graph}
\end{figure}

If we label $[n,n+1)$ as ``interval $n$'' for $n \in \mathbb{Z}$, then the function $f_\sigma$ has the effect of permuting these half-open intervals as prescribed by $\sigma$. For example, Fig. \ref{graph} illustrates the graph of the function corresponding to the permutation $\sigma = (012534)$, where we assume that $\sigma$ fixes all other integers. It is clear that composing such functions commutes with composition of the underlying permutations; in other words, if $\sigma,\tau:\mathbb{Z} \to \mathbb{Z}$ are permutations, then $f_{\sigma\tau} = f_\sigma\circ f_\tau$. In particular, $f_\sigma \circ f_{\sigma^{-1}}$ is the identity, so $f_\sigma^{-1} = f_{\sigma^{-1}}$. Therefore the group generated by $f_\sigma$ and $f_\tau$ will be free as long as no composition involving $\sigma, \tau, \sigma^{-1},\tau^{-1}$ that is reduced (meaning no permutation is adjacent to its inverse) is equal to the identity.

Beyond merely generating a free group, we seek a pair of functions for which the elements of the resulting free group have few fixed points. In the standard Banach-Tarski construction the nontrivial free group elements are all rotations, each having two fixed points at the poles where the axis of rotation meets the surface of the sphere. These points are a nuisance -- they must be handled separately and increase the number of pieces required for the construction beyond the basic four. In light of this, we are interested in finding $\sigma$ and $\tau$ for which no composition involving $\sigma,\tau,\sigma^{-1},\tau^{-1}$ even has a fixed point, a much stronger condition than the one required to obtain a free group.

Momentarily we will exhibit such a pair of permutations $\sigma$ and $\tau$. Observe that neither permutation can have a cycle of finite length, for if $\sigma$ had a cycle of length $n$, for example, then $\sigma^n$ would have $n$ fixed points. Neither can $\sigma$ consist of a single infinite cycle; for if $\tau(0)=a$ then some power of $\sigma$ or $\sigma^{-1}$ will map $a$ to 0, meaning that $\sigma^n\tau$ has a fixed point for some $n \in \mathbb{Z}$. In fact, further consideration shows that if either $\sigma$ or $\tau$ consists of only finitely many infinite cycles then a fixed point is again inevitable. So $\sigma$ and $\tau$ must be constructed with some care.

\begin{figure}
\includegraphics[scale=.5]{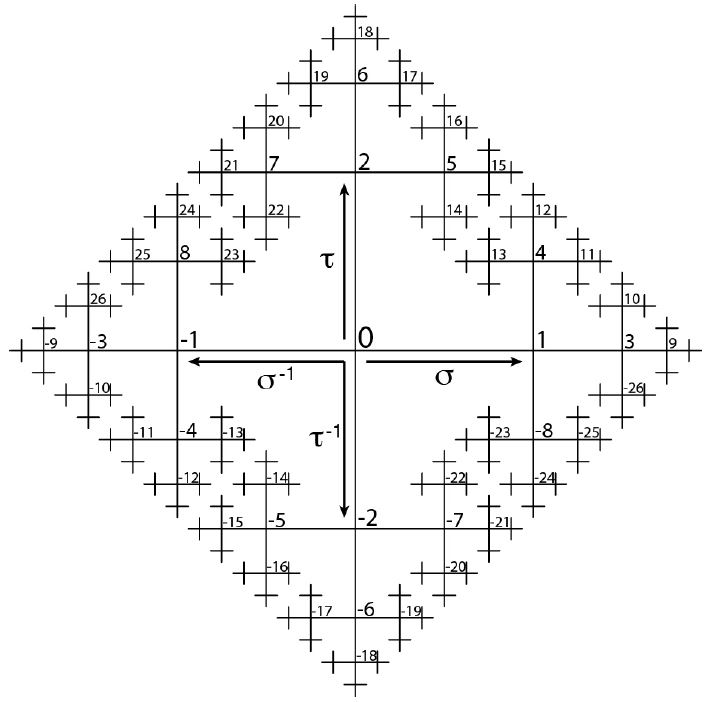}
\caption{The labeled Cayley graph used to define $\sigma$ and $\tau$.}
\label{cay}
\end{figure}

\begin{proposition}
There exist permutations $\sigma,\tau:\mathbb{Z} \to \mathbb{Z}$ that generate a free subgroup within the group of all permutations of $\mathbb{Z}$, with the property that no element of this free group, other than the identity, has any fixed points.
\end{proposition}

\begin{proof}
We will construct such permutations using the Cayley graph for a free group on two generators. Recall that this graph contains countably many vertices, each with degree four, typically displayed with horizontal and vertical edges, as in Fig. \ref{cay}. We have labeled the vertices with all the integers; any particular bijection between $\mathbb{Z}$ and the vertex set will suffice for the construction. For each horizontal edge having vertices labeled $m$ and $n$ from left to right, we now declare $\sigma(m)=n$. From a dynamic point of view, $\sigma$ maps each integer to the one appearing immediately to its right in the Cayley graph, while $\sigma^{-1}$ maps each integer to the one directly to the left. Similarly, $\tau(m)$ and $\tau^{-1}(m)$ are defined as the integers just above and below $m$, respectively. Thus $\sigma(2)=5$ and $\tau^{-1}(8)=-1$, for instance.

Therefore the permutation given by a reduced word on $\sigma$ and $\tau$ will map each integer to an image that may be reached via a certain fixed sequence of steps through the Cayley graph, each step being left, right, up or down. Since the word is reduced the path will never double back on itself, and since the Cayley graph is a tree the image cannot be equal to the original integer. In other words, no composition of $\sigma$ and $\tau$ can have any fixed points, since mapping an integer to itself corresponds to a cycle through the tree, which is impossible.
\end{proof}

In fact we have shown slightly more; namely, given any distinct $m,n \in \mathbb{Z}$ there exists a unique reduced composition of $\sigma$ and $\tau$ mapping $m$ to $n$. For instance, to map $-14$ to 6 we simply trace a path through the tree from $-14$ down to 5, right to $-2$, then up to 6, revealing that $\tau^3(\sigma(\tau^{-1}(-14))))=6$.

\section{Describing the Decomposition}

We are now in possession of a pair of functions $f_\sigma,f_\tau:\mathbb{R} \to \mathbb{R}$ that are piecewise rigid, generate a free group, and have the property that no reduced composition of these functions has any fixed points. In what follows we will write $g=f_\sigma$ and $h=f_\tau$ to simplify notation. The remainder of the argument is standard, although considerably simpler than the original proof involving spheres, since there are no ``bad points'' to handle separately, nor is the axiom of choice required. We follow the approach described in \cite{ZU}.

\begin{theorem}
It is possible to partition the real number line $\mathbb{R}$ into four subsets $A,B,C$ and $D$ such that
\[A \sqcup g(B) = C\sqcup h(D) = \mathbb{R},\]
where $g,h:\mathbb{R} \to \mathbb{R}$ are piecewise rigid functions and $\sqcup$ denotes disjoint union.
\end{theorem}

\begin{proof}
Let $g=f_\sigma$ and $h=f_\tau$ be the functions constructed above, and let $G$ be the free group generated by $g$ and $h$. Then an element of $G$ has the form
\[g^{a_1}h^{b_1}g^{a_2}\hdots h^{b_{k-1}}g^{a_k}h^{b_k}, \textrm{ with }a_i,b_j \in \mathbb{Z}\textrm{ for }1\leq i,j\leq k,\]
where we require exponents to be nonzero except for possibly $a_1$ or $b_k$, to allow for compositions that begin with a power of $h$ or end with a power of $g$.

We next partition the elements of $G$ into four subsets. To begin, let $G_1$ consist of all elements beginning with a positive power of $g$ and $G_2$ contain all elements beginning with a negative power of $g$. One can then verify that $G=G_1 \sqcup g(G_2)$, where $g(G_2)$ is the set $\{g \circ f|f\in G_2\}$.

It would be convenient to create similar sets according to the initial power of $h$, but then the identity would not appear in any subset. So we instead let $G_3$ consist of all elements beginning with a positive power of $h$ followed by at least one more term, and define $G_4$ as the set containing any element beginning with a negative power of $h$, along with the identity and elements of the form $h^b$ with $b \in \mathbb{N}$. We then have $G = G_3\sqcup h(G_4)$ by construction.

Now label $I = [0,1)$ and set $A = G_1(I),B=G_2(I),C=G_3(I)$ and $D = G_4(I)$, where
\[G_1(I)=\bigcup_{f\in G_1}f(I),\]
and similarly for the other sets. The fact that there exists a unique permutation generated by $\sigma$ and $\tau$ mapping 0 to any particular integer ensures that the intervals $f(I)$ for $f \in G$ are distinct and together comprise all of $\mathbb{R}$. Because the sets $G_1$ through $G_4$ partition $G$, we conclude that $\mathbb{R} = A\sqcup B\sqcup C\sqcup D$ as well. Finally, the fact that $G = G_1\sqcup g(G_2)$ translates to $\mathbb{R} = A \sqcup g(B)$, and we similarly find that $\mathbb{R} = C\sqcup h(D)$, as desired.
\end{proof}

A paradoxical decomposition of a set typically has a consequence regarding what sorts of measures are possible on that set. In our case we have the following.

\begin{cor}
There is no finitely additive measure $\mu$ on the real number line with $\mu(\mathbb{R}) \in (0,\infty)$ that is also invariant under piecewise rigid maps.
\end{cor}
\begin{proof}
Suppose such a measure $\mu$ existed, say with $\mu(\mathbb{R})=m>0$. Then
\[m = \mu(A \sqcup g(B)) = \mu(A)+\mu(B)\textrm{ and }m=\mu(C\sqcup h(D)) = \mu(C)+\mu(D).\]
But we also know that 
\[m=\mu(A\sqcup B\sqcup C \sqcup D) = \mu(A)+\mu(B)+\mu(C)+\mu(D),\]
which implies that $m=2m$, a contradiction. Hence no such measure exists.
\end{proof}

This result is of interest given the existence of a finitely additive measure on the Lebesgue subsets of the real number line that is invariant with respect to isometries and for which $\mu(\mathbb{R})=1$, as constructed in \cite{M79}.

\section{Extensions}

It does not take a great leap of imagination to realize that a similar construction may be effected that will permit us to split the real number line into a disjoint union of $2k$ sets for any $k \in \mathbb{N}$, then reassemble them via piecewise rigid functions into $k$ copies of the line. One simply employs the Cayley graph for a free group on $k$ generators, labels the vertices with all the integers as before, defines permutations $\sigma_1$ through $\sigma_k$ which map integers in each of the $k$ ``positive directions'' within the graph, and subsequently defines functions $f_1$ through $f_k$ based on these permutations.

There is a paradoxical decomposition of the free group $G$ on $k$ generators into $k$ pairs of subsets analogous to before, in which $k-1$ of the pairs consist of words beginning with a positive or negative power of some $f_j$, say for $2 \leq j \leq k$, while the last pair involving $f_1$ utilizes the small variation that allows us to include the identity of $G$ among the subsets. The various subsets of $\mathbb{R}$ are then given by the images of $I=[0,1)$ as before, and we obtain the paradoxical decomposition in the same manner.

In fact, there is nothing to prevent one from carrying out the same construction beginning with the Cayley graph for a free group on countably many generators, since the vertex set of this graph is still countably infinite and thus may be labeled with all the integers. The remainder of the process goes through with hardly any modification. In summary, we have the following result.

\begin{theorem}
Let $k \geq 2$ be a positive integer. Then it is possible to partition the real number line $\mathbb{R}$ into $2k$ subsets $A_1,B_1,\hdots,A_k,B_k$ such that
\[A_1\sqcup f_1(B_1) = \hdots = A_k \sqcup f_k(B_k) = \mathbb{R},\]
for piecewise rigid functions $f_1,\hdots,f_k$. Furthermore, one can partition $\mathbb{R}$ into countably many subsets $A_1,B_1,A_2,B_2,\hdots$ such that
\[A_1\sqcup f_1(B_1) = A_2 \sqcup f_2(B_2) = \hdots = \mathbb{R},\]
where all $f_j$ are piecewise rigid functions.
\end{theorem}

We remark in closing that our approach relies upon the fact that the real number line is unbounded; an analogous sort of construction could not be performed to obtain a paradoxical decomposition of an interval, for instance, at least not using piecewise rigid functions.
\newpage

\section{Acknowledgements}
The authors are grateful to Japheth Wood for a conversation that led to a considerably more streamlined construction of $\sigma$ and $\tau$ than was originally envisioned. Thanks also to Paul Zeitz for sharing his carefully crafted course notes as well as to Stan Wagon for observing that our decomposition leads to implications regarding measures on the real number line.


\end{document}